\documentclass[11pt]{amsart}

\usepackage{pdfsync}
\usepackage{color}

\usepackage{graphicx}
\usepackage{amssymb,amsfonts,amsthm}
\usepackage{amsmath,amsthm,amssymb}
\usepackage{amscd}
\usepackage{amsfonts}
\usepackage[all]{xy}

\theoremstyle{definition}
\theoremstyle{plain}
\newtheorem{thm}{Theorem}[section]

\newtheorem{lem}[thm]{Lemma}
\newtheorem{prop}[thm]{Proposition}

\theoremstyle{definition}

\newtheorem{rmk}[thm]{Remark}

\newcommand{\mb}[1]{\mathbb{#1}}

\renewcommand{\phi}{\varphi}

\makeatletter
\def\ulabel#1#2{\@bsphack\if@filesw {\let\thepage\relax \def\protect{\noexpand\noexpand\noexpand}%
\xdef\@gtempa{\write\@auxout{\string
\newlabel{#1}{{#2 \@currentlabel}{\thepage}}}}}\@gtempa
\if@nobreak \ifvmode\nobreak\fi\fi\fi\@esphack} \makeatother

\title{Inverse limits of finite rank free groups}
\author[G.R. Conner]{G. R. Conner}
\address{G.R. Conner\\ Department of Mathematics\\ Brigham Young University\\292-A TMCB\\
Provo UT 84602; USA}
\email{conner@math.byu.edu}

\author[C. Kent]{C. Kent}
\address{C. Kent\\ Department of Mathematics\\ Vanderbilt University\\Stevenson Center 1212
Nashville TN 37240; USA}
\email{curt.kent@vanderbilt.edu}

\begin{document}

\maketitle

\begin{abstract}

We will show that the inverse limit of finite rank free groups with surjective connecting homomorphism is isomorphic either to a finite rank free group or to a fixed universal group.  In other words, any inverse system of finite rank free groups which is not equivalent to an eventually constant system has the universal group as its limit. This universal inverse limit is naturally isomorphic
to the first shape group of the Hawaiian earring. We also give an example
of a homomorphic image of a Hawaiian earring group which lies in the inverse limit of free groups but is neither a free group
nor a Hawaiian earring group.\end{abstract}

One of the first to consider the inverse limits of finite rank free groups
was Higman.  In \cite{hig}, he studies the inverse limit of finite rank
free groups which he calls  the unrestricted free product of countably
many copies of $\mb Z$. There he proves that this group is not a free
group and that each of its free quotients has finite rank. He considers a
subgroup $P$ of the unrestricted product which turns out to be a Hawaiian
earring group but does not prove it there. In \cite{smit}, de Smit gives a
proof that the Hawaiian earring group embeds in an inverse limit of free
groups and gives a characterization of the elements of the image.
Daverman and Venema in \cite{DV} showed that a one-dimensional Peano
continuum either has the shape of a finite bouquet of circles or of a
Hawaiian earring.   Hence any inverse limit of finite rank free groups
that arises from the inverse system of a one-dimensional Peano continuum
is either a finite rank free group or the standard Hawaiian inverse limit.
In section \ref{seclimit}, we will show that the result of Daverman and
Venema can be generalized in the following way.  Every inverse limit of
finite rank free groups with surjective connecting homomorphisms is isomorphic to a free group or the standard Hawaiian inverse limit.  Hence being the shape group of a one-dimensional Peano continuum is not necessary.

In \cite{ce}, Conner and Eda show that the fundamental group of a
one-dimensional Peano continuum which is not semilocally simply connected
at any point determines the homotopy type of the space.  This was done
using uncountable homomorphic images of Hawaiian earring groups. It was
believed that any uncountable homomorphic image of a Hawaiian earring
group which embedded in an inverse limit of free groups was itself a
Hawaiian earring group.   In Section \ref{secimage}, we will show that
there exists an uncountable homomorphic image of a Hawaiian earring group
which embeds in an inverse limit of free groups but is not a Hawaiian
earring group or a free group.  This is done by using two propositions
which were originally proved by Higman in \cite{hig} to construct a
homomorphism with our desired image.

\section{Definitions}

A Hawaiian Earring group, which we will denote by $\mathbb H $,
is the fundamental group of the one-point compactification of a
sequences of disjoint open arcs. The Hawaiian earring group is uncountable and locally free. Cannon and Conner in \cite {cc1} and \cite{cc3} showed that the Hawaiian earring group is generated in the sense of infinite products by a countable sequence of loops corresponding to the disjoint arcs, where an infinite product is legal if each loop is transversed only finitely many times. (For more information on infinite products, see \cite{cc3}.) The Hawaiian earring can be realized in the plane as the union of circles centered at $(0,\frac1n)$ with radius $\frac1n$.  We will use $\textbf{E}$ to denote this subspace of the plane and $a_n$ to denote the circle centered at $(0,\frac1n)$ with radius $\frac1n$.

The group $\mb H$ is generated, in the sense of infinite products, by an infinite set of loops which correspond to the circles $\{a_n\}$. When there is no chance of confusion, we will refer to this infinite generating set for the fundamental group of $\textbf{E}$ as $\{a_n\}$, i.e. $a_n$ represents the loop which transverses counterclockwise one time the circle of radius $\frac 1n$ centered at $(0,\frac1n)$.  We will frequently denote the base point $(0,0)$ of $\textbf{E}$ by just $0$.

An \emph{inverse system of groups} is a collection of groups $F_\alpha$ indexed by a partially ordered set $J$ along with a collection of homomorphisms $\{\phi_{\alpha,\beta}:F_\alpha\to F_\beta \ |\ \text{if } \alpha\geq\beta\}$, which are called \emph{connecting homomorphisms}.  The connecting homomorphisms must satisfy the following condition.  For every triple $\{\alpha, \beta, \gamma\}$ such that $\alpha\geq\beta\geq\gamma$, the connecting homomorphisms satisfy $\phi_{\alpha,\beta}\circ\phi_{\beta,\gamma} = \phi_{\alpha,\gamma}$.  An inverse limit of an inverse system is the subgroup of the direct product which consists of functions $f: J\to \bigcup\limits_{\alpha\in J} F_\alpha$  such that $f(\beta)\in F_\beta$ and $f(\beta) =\phi_{\alpha,\beta}(f(\alpha))$, for every $\alpha,\beta$ such that $\beta \leq \alpha$.

\section{Inverse Limits of Finite Rank Free groups}\label{seclimit}

We will now describe the inverse system of finite rank free groups constructed by Higman. Let $A_i=\langle \mathbf{a}_1, \cdots \mathbf{a}_i\rangle$ be the free group on $i$ generators with the natural inclusions $A_1\subset  \cdots \subset A_i\subset A_{i+1}\subset\cdots$.  For $i\geq j$, the connecting homomorphisms $P_{i,j}: A_i \to A_j$ sends $\mathbf a_k$ to $\mathbf a_k$ for $k\leq j$ and $\mathbf a_k$ to $1$ for $k>j$.  We will denote this inverse limit of this system by $\mb G$.  We will use $P_i: \mb G \to A_i$ to denote the standard projection homomorphism.

Eda \cite{edaprivate} pointed out to the authors that Proposition \ref{a} doesn't hold in the case that the connecting homomorphisms are not surjective.  However, it turns out the the proof of Proposition \ref{a} is still sufficient to enumerate all possible isomorphism types of inverse limits of countable rank free groups (see Remark \ref{remark2}).

\begin{prop}\label{a}
Let $G$ be an inverse limit of finite rank free groups, $F_i$, with surjective connected homomorphisms indexed over the natural numbers.  Then $G$ is isomorphic to $\mb G$ or a finite rank free group.
\end{prop}

The following two lemmas are well known.

\begin{lem}\label{2}
If $G$ is an inverse limit of groups $\{F_\alpha\}$ and $G'$ is the
inverse limit of some cofinal sequence of $\{F_\alpha\}$, then $G$ is
isomorphic to $G'$.
\end{lem}



\begin{lem}\label{remark}
Any morphism of inverse systems which consists of isomorphisms induces an isomorphism of limits.
\end{lem}

\begin{proof}[Proof of Proposition \ref{a}]
Let $\pi_{i,j}$ be the connecting homomorphisms of $G$.  If the rank of $F_n$ is eventually constant then the connecting homomorphisms must eventually be isomorphisms (see Proposition 2.12 in \cite{LS}).  Hence $G$ is a finite rank free group. Otherwise by passing to a cofinal subsequence, we may assume that the rank of $F_n$ is a strictly increasing sequence.

Let $B_1$ be a basis for $F_1$.  By induction, suppose that for all $m<n$, $B_m\cup K_m$ is a basis for $F_m$ such that $\langle K_m\rangle\subset \ker(\pi_{m,m-1})$ and $\pi_{m,m-1}$ maps $B_m$ bijectively onto $B_{m-1}\cup K_{m-1}$.  

There exists a free basis $B_n'\cup K_n$ of $F_n$ with the property that $\langle K_n\rangle\subset \ker(\pi_{n,n-1})$ and $\pi_{n,n-1}$ restricted to $\langle B_n'\rangle$ is an isomorphism (again by Proposition 2.12 in \cite{LS}).  Now restricting $\pi_{n,n-1}$ to this set we may define  $B_n = \bigl(\pi_{n,n-1}|_{\langle B_n'\rangle}\bigr)^{-1}(B_{n-1})$.  Since $\pi_{n,n-1}$ restricted to $\langle B_n'\rangle$ is an isomorphism; $B_n$ is a free basis
for $\langle B_n'\rangle$.  It is a simple exercise to show that $B_n\cup K_n$
is still a free basis for $F_n$.

It is now trivial to find isomorphisms between the groups $F_n$ and $A_{|B_n\cup K_n|}$ which commute with the connecting homomorphisms.  The result follows from the Lemma \ref{remark} and Lemma \ref{2}.

\end{proof}

\begin{rmk}\label{remark2}Given a system of countable rank free groups $(F_n, \pi_{n,n-1})$ with inverse limit $G$, one can pass to the inverse system $(\pi_n(G), \pi_{n,n-1})$ without changing isomorphism types of the inverse limit where $\pi_n$ is the canonical projection of $G$ to $F_n$.  The new system then has surjective connecting homomorphisms.  However, $\pi_n(G)$ is a free group of possibly infinite rank.  Then as in the proof of Theorem \ref{a}, we can find a basis $B_n\cup K_n$ of $\pi_n(G)$ with the same properties as before.  After passing to a cofinal inverse system, we may assume that $B_n$, $K_n$ are eventually trivial, finite, or countably infinite.  Then the isomorphism type of the inverse limit is determined by the cardinalities of $B_n$, $K_n$.    This gives five possible isomorphism types.  Eda \cite{edaprivate} shows that these five types can all be realized and are distinct.
\end{rmk}

\section{Images of the Hawaiian earring group}\label{secimage}

Note that $A_i$ embeds in $\mb G$ by sending $\mathbf{a}_i $ to the element which is $1$ in the first $i-1$ coordinates and $\mathbf{a}_i$ in all other coordinates.  The proof of de Smit in \cite{smit} shows the embedding of $\mb H$ into $\mb G$ sends $a_i$ to $\mathbf a_i$.

We will give $\mb G$ a metric such that $\mb G$ under the induced
topology  is a topological group. In \cite{hig}, Higman defines a topology which is equivalent to our metric topology.  Let $d_i$ be the
$(0,1)$-metric on $B_i$, i.e. $d_i(x,y)=0$ if $x=y$ and $1$
otherwise.  For $(g_n),(h_n)\in\mb G$, let $d((g_n),(h_n)) =
\sum\limits_{n} \frac{1}{2^n}d_n(g_n,h_n)$.

\begin{rmk}\ulabel{converges}{Remark}
Under this topology, $g_i=(g_n^i)$ converges to $h=(h_n)$ if and only if for each $n$ there exists an $M(n)$ such that $g_n^i h_n^{-1}=1$ for all $i\geq M(n)$.  It follows that if $g_i$ converges $g_ig_{i+1}^{-1}$ must converge to $1$.  Suppose that $g_ig_{i+1}^{-1}$ converges to $1$.  Then $g_n^i$ considered as a sequence in $i$ is eventually constant for every $n$.  Hence $g_i$ converges.

\end{rmk}

We will leave it to the reader to verify that this topology makes $\mb G$ into a topological group.  An interesting note is that under this metric $\mb G$ is complete and all three of the sets $\mb H$ , $\mb G -\mb H$, and the free group generated by $\{\mathbf a_i \ | \ i\in \mb N\}$ which wee will denote by $\langle\mathbf a_1,\mathbf{a}_2, \cdots\rangle$ are dense.

The following two propositions of Higman demonstrate the elegance of this topology.  For completeness and to make the proof readily accessible to the reader, we will include their proofs here.

\begin{prop}Any endomorphism of $\mb G$  is continuous. \end{prop}

\begin{proof}
Let $\phi: \mb G\to\mb G$ be an endomorphism.  Then $d((g_n),(h_n))\leq \frac{1}{2^i}$ if and only if $g_n= h_n$ for all $n\leq i$.  Higman in \cite{hig} showed that $P_i\circ\phi$ factors through some $P_{n(i)}$.  (see Theorem 1)  Cannon and Conner have shown that the same holds true for any $\phi: \mb H \to F$ where $F$ is a free group. (see Theorem 4.4 in \cite{cc3}) Hence, $P_i\circ\phi = \phi\circ P_{n(i)}$ for some $n(i)$ which depends on $i$.  Thus $d((g_n),(h_n)) \leq \frac{1}{2^{n(i)}}$ implies that $d(\phi((g_n)),\phi((h_n))) \leq \frac{1}{2^{i}}$.
\end{proof}

\begin{prop}
Any set function $\phi$ from $\{\mathbf {a}_1,\mathbf{a}_2,\cdots \}$ to $\mb G$ such that $d(\phi(\mathbf a_i),1)$ converges to $0$ extends to an endomorphism of $\mb G$.
\end{prop}

\begin{proof}
Let $\phi:\{\mathbf {a}_1,\mathbf{a}_2,\cdots \}\to \mb G$ be a set function such that $d(\phi(\mathbf a_i),1)$ converges to $0$.  Then $\phi$ will extend to the free group $\langle\mathbf a_1,\mathbf{a}_2, \cdots\rangle\leq \mb G$.  We want to be able to extend $\phi$ to all of $\mb G$.

Suppose that $g_i\in \langle\mathbf a_1,\mathbf{a}_2, \cdots\rangle$ converges in $ \mb G$.  We must show that $\phi(g_i)$ also converges.   By \ref{converges}, it is enough to show that $\phi(g_i)\phi(g_{i+1})^{-1}$ converges.  Since $\phi$ is a homomorphism on $\langle\mathbf a_1,\mathbf{a}_2, \cdots\rangle$, $\phi(g_i)\phi(g_{i+1})^{-1}= \phi(g_ig_{i+1}^{-1})$.  Hence it is sufficient to show that if $g_i\to1$ then $\phi(g_i)\to 1$.

Let $\phi(\mathbf a_i) = (\phi(\mathbf a_i)_n)$.  Suppose that $g_i = (g^i_n)$ converges to $1$.  Fix $n$.  We will show that $\phi(g_i)_n$ when considered as a sequence in $i$ is eventually trivial.  Fix $M$ such that $\phi(\mathbf a_i)_n = 1$ for all $i\geq M$.  Fix $M'$ such that $g^i_n = 1$ for all $i\geq M'$.  By construction, $g_i$ is a word $w_i(\{\mathbf a_j\})$ in $\langle\mathbf a_1,\mathbf{a}_2, \cdots\rangle$.  Then $\phi(g_i) = w_i(\{\phi(\mathbf a_j)\})$; hence, $\phi(g_i)_n = w_i(\{\phi(\mathbf a_j)_n\})$.  Then for all $i\geq \max\{M,M'\}$, $\phi(g_i)_n = w_i(\{\phi(\mathbf a_j)_n\})= 1$.

Suppose that $g_i, g_i'\in \langle\mathbf a_1,\mathbf{a}_2, \cdots\rangle$ converge to the same element of $ \mb G$.  Then $g_i'g_i^{-1}$ converges to $1$.  Then $\phi(g_i')\phi(g_i^{-1}) = \phi(g_i'g_i^{-1})\to 1$.  Thus $\phi$ extends to a well defined continuous function $\overline\phi:\mb G\to\mb G$ which is independent of the chosen sequence.

Suppose that $g,h\in\mb G$.  Then there exists $g_i,h_i \in \langle\mathbf a_1,\mathbf{a}_2, \cdots\rangle$ such that $g_i\to g$ and $h_i\to h$.  Since $\overline\phi$ is independent of the chosen sequence, $\phi(g_i)\phi(h_i) \to \overline\phi(g)\overline\phi(h)$ and $\phi(g_i)\phi(h_i) = \phi(g_ih_i) \to \overline\phi(gh)$.  Hence $\phi$ extends to a homomorphism $\overline\phi$.
\end{proof}

For a path $\alpha: [0,t]\to X$, we will also use $\overline\alpha$ to represent the path where $\overline\alpha(s) = \alpha(t-s)$ . We will also use the following theorem.

\begin{thm}\ulabel{cont}{Theorem}[Eda \cite{eda}]
Let $\psi:\mathbb H \to\pi_1(X,x_0)$ a homomorphism into the fundamental group of a
one-dimensional Peano continuum $X$. Then there exists a continuous
function $f:(\textbf{E},0) \to (X,x)$ and a path $\alpha:(I,0,1)\to (X,x_0,x)$, with the property that $f_* =\widehat\alpha\circ\phi$.  Additionally, if the image of $\psi$ is uncountable the $\alpha$ is unique up to homotopy rel endpoints.
\end{thm}

Another proof, as well as a proof for a planar version, of this theorem can be found in the Masters Thesis of the second author (see \cite{ck}).  Cannon and Conner in \cite {cc3} showed that in one dimensional spaces there exists a unique (up to reparametrization) reduced representative for each path class.  We will use $[\cdot]_r$ (or $\phi(\cdot)_r$) to represent the unique reduced representative for the path class $[\cdot]$ (or  $\phi(\cdot)$).

We are now ready to give our counter example.  We will begin by defining a set function $\phi:\{\mathbf {a}_1,\mathbf{a}_2,\cdots \}\to \mb G$ by
 $$\phi(\mathbf{a}_i)= \begin{cases} \mathbf{a}_{i} &   \text{if $i$ is odd}
\\ \mathbf{a}_1\mathbf{a}_{i}{\mathbf{a}_1}^{-1} &   \text{if $i=2\mod4$} \\ \mathbf{a}_{i-2} &
\text{if $i=0\mod4$} \end{cases}$$

Then $\phi$ extends to an endomorphism of $\mb G$ which we may then restrict to the naturally embedded $\mb H$.  Thus after extending and restricting, $\phi: \mb H \to \mb G$ is a homomorphism with uncountable image.  Suppose there existed an isomorphism $\psi: \phi(\mb H) \to \pi_1(\textbf{E},0)$.  Then $\psi\circ\phi: \mb H\to \pi_1(\textbf{E},0)$ is a homomorphism from $\mb H$ to a one-dimensional Peano continuum which by \ref{cont} must be conjugate to a homomorphism induced by a continuous function.  Let $T$ the be the path such that $\widehat{T}\circ \phi$ is induced by a continuous function.   Then by construction $\psi\circ\phi(a_{4i-2})= \psi(a_1)^{-1}\psi(\mathbf{a}_{4i-2})\psi(a_1)$ and $\psi\circ\phi(a_{4i})= \psi(\mathbf{a}_{4i-2})$ which imply that

\begin{align*}[\overline T* \psi\circ\phi(a_{4i-2})_r*T] &= [\overline T*\bigl(\psi(a_1)^{-1}\psi(\mathbf{a}_{4i-2})\psi(a_1)\bigr)_r*T] \\ &= [\overline{T}*\bigl(\psi(a_1)^{-1}\psi\circ\phi(a_{4i})\psi(a_1)\bigr)_r*T]\\ &= [\overline{T}*\overline{\psi(a_1)_r}*T*\overline T* \psi\circ\phi(a_{4i})_r *T * \overline T *\psi(a_1)_r*T] \\ &= [\bigl(\overline{\overline T*\psi(a_1)_r*T}\bigr)*\overline T* \psi\circ\phi(a_{4i})_r *T * \bigl(\overline {T} *\psi(a_1)_r*T\bigr)] \\ &= [\overline{\overline T*\psi(a_1)_r*T}]\cdot[\overline T* \psi\circ\phi(a_{4i})_r *T ]\cdot[\overline {T} *\psi(a_1)_r*T]\end{align*}

By our choice of $T$, $\{[\overline T* \psi\circ\phi(a_{4i-2})_r*T]_r\}$ and $\{[\overline T* \psi\circ\phi(a_{4i})_r*T]_r\}$ are null sequences of loops in $\textbf{E}$. However, the second sequence of loops is conjugate to the first by a non-trivial loop $[\overline {T} *\psi(a_1)_r*T]_r$ which is a  contradiction.

\bibliographystyle{plain}
\bibliography{planarbib}

\begin{thebibliography}{10}

\bibitem{cc1}
J.~W. Cannon and G.~R. Conner.
\newblock The combinatorial structure of the {H}awaiian earring group.
\newblock {\em Topology Appl.}, 106(3):225--271, 2000.

\bibitem{cc3}
J.~W. Cannon and G.~R. Conner.
\newblock On the fundamental groups of one-dimensional spaces.
\newblock {\em Topology Appl.}, 153(14):2648--2672, 2006.

\bibitem{ce}
Greg Conner and Katsuya Eda.
\newblock Fundamental groups having the whole information of spaces.
\newblock {\em Topology Appl.}, 146/147:317--328, 2005.

\bibitem{DV}
R.~J. Daverman and G.~A. Venema.
\newblock C{E} equivalence and shape equivalence of {$1$}-dimensional compacta.
\newblock {\em Topology Appl.}, 26(2):131--142, 1987.

\bibitem{smit}
Bart de~Smit.
\newblock The fundamental group of the {H}awaiian earring is not free.
\newblock {\em Internat. J. Algebra Comput.}, 2(1):33--37, 1992.

\bibitem{edaprivate}
Katsuya Eda.
\newblock Private communication.

\bibitem{eda}
Katsuya Eda.
\newblock The fundamental groups of one-dimensional spaces and spatial
  homomorphisms.
\newblock {\em Topology Appl.}, 123(3):479--505, 2002.

\bibitem{hig}
Graham Higman.
\newblock Unrestricted free products, and varieties of topological groups.
\newblock {\em J. London Math. Soc.}, 27:73--81, 1952.

\bibitem{ck}
C.~Kent.
\newblock {\em Homomorphisms into the fundamental group of one-dimensional and
  planar Peano continua}.
\newblock Masters Thesis. Brigham Young University, 2008.

\bibitem{LS}
Roger~C. Lyndon and Paul~E. Schupp.
\newblock {\em Combinatorial group theory}.
\newblock Classics in Mathematics. Springer-Verlag, Berlin, 2001.
\newblock Reprint of the 1977 edition.

\end{thebibliography}

\end{document}